\documentclass[12pt]{article}
\usepackage{amsmath, amsthm, amscd, amsfonts, amssymb, graphicx}
\usepackage{amsmath,amssymb,latexsym}
\usepackage{mathtools}
\usepackage[T1]{fontenc}
\usepackage{graphicx,tikz,color,url}
\usetikzlibrary{decorations.pathreplacing}
\usepackage{comment}
\usepackage{enumitem}
\usepackage{multicol}

\newtheorem{theorem}{Theorem}[section]
\newtheorem{lemma}[theorem]{Lemma}
\newtheorem{corollary}[theorem]{Corollary}
\newtheorem{proposition}[theorem]{Proposition}

\newtheorem{problem}[theorem]{Problem}

\usepackage[normalem]{ulem}

\title{Complexity and Structural Results for the Hull and Convexity Numbers in Cycle Convexity for Graph Products}
\author{Bijo S. Anand$^{a}$\footnote{bijos\_anand@yahoo.com}
\and
Ullas Chandran S. V. $^{b}$\footnote{svuc.math@gmail.com}
\and
Julliano R. Nascimento$^{c}$\footnote{jullianonascimento@ufg.br}
\and Revathy S. Nair$^{d}$\footnote{revathyrahulnivi@gmail.com} 
\\\\
$^{a}$\small Department of Mathematics, Sree Narayana College, Punalur, Kerala\\
$^{b}$\small Department of Mathematics, Mahatma Gandhi College,\\\small Thiruvananthapuram, Kerala, India\\
$^{c}$ \small Instituto de Informática, Universidade Federal de Goiás, Goiânia, GO, Brazil\\
$^{d}$ \small Department of Mathematics, Mar Ivanios College, University of Kerala, \\\small Thiruvananthapuram, India
}
\date{\today}

\begin{document}
\maketitle
\begin{abstract}
Let $G$ be a graph and $S \subseteq V(G)$. In the cycle convexity, we say that $S$ is \textit{cycle convex} if for any $u\in V(G)\setminus S$, the induced subgraph of $S\cup\{u\}$ contains no cycle that includes $u$. The \textit{cycle convex hull} of $S$ is the smallest convex set containing $S$. The \textit{cycle hull number} of $G$, denoted by $hn_{cc}(G)$, is the cardinality of the smallest set $S$ such that the convex hull of $S$ is $V(G)$.
The \textit{convexity number} of $G$, denoted by $C_{cc}(G)$, is the maximum cardinality of a proper convex set of $V(G)$.
This paper studies cycle convexity in graph products. We show that the cycle hull number is always two for strong and lexicographic products. For the Cartesian, we establish tight bounds for this product and provide a closed formula when the factors are trees, generalizing an existing result for grid graphs. In addition, given a graph $G$ and an integer $k$, we prove that $hn_{cc}(G) \leq k$ is NP-complete even if $G$ is a bipartite Cartesian product graph, addressing an open question in the literature. Furthermore, we present exact formulas for the cycle convexity number in those three graph products. That leads to the NP-completeness of, given a graph $G$ and an integer $k$, deciding whether $C_{cc}(G) \geq k$, when $G$ is a Cartesian, strong or lexicographic product graph.
\end{abstract}

\noindent{\small {\bf Keywords:} convexity; convexity number; hull number; cycle convexity number; cycle hull number; Cartesian product; strong product; lexicographic product.}

\noindent{\small {\bf AMS Subj.Class:} 05C69, 05C76, 05C85}

\section{Introduction} \label{sec:intro}
A \emph{finite convexity space} is defined as a pair $(V, \mathcal{C})$, where $V$ is a non-empty finite set and $\mathcal{C}$ is a collection of subsets of $V$ satisfying the following properties: $\emptyset \in \mathcal{C}$, $V \in \mathcal{C}$, and $\mathcal{C}$ is closed under intersections. The elements of $\mathcal{C}$ are referred to as \emph{convex sets}; see \cite{van-1993}. Various convexity structures associated with the vertex set of a graph are widely studied. The most natural convexities in graphs are path convexities, defined in terms of a family of paths $\mathcal{P}$, where a set $S$ of vertices in $G$ is $\mathcal{P}$-\emph{convex} if $S$ includes all vertices of every path in $\mathcal{P}$ between any two vertices of $S$. An extensive overview of different types of path convexities are provided in \cite{pelayo-2015}. The well-known \emph{geodesic convexity} corresponds to the case where $\mathcal{P}$ is the family of all shortest paths; see \cite{everett1985hull, buckley-1990, farber-1986}.
Other significant examples include \emph{monophonic convexity} \cite{caceres-2005, source17, source20} and $P_3$-\emph{convexity}~\cite{source11, centeno2011irreversible, coelho2019p3}. 
These convexities are defined respectively over induced paths and paths with three vertices. There are also convexity definitions that do not rely on path systems; some significant examples include \emph{Steiner convexity}~\cite{source9} and $\Delta$-\emph{convexity}~\cite{bijo2,bijo3,bijo1}. In the case of Steiner convexity, a set $S \subseteq V(G)$ is \emph{Steiner convex} if, for any subset $S' \subseteq S$, all vertices of any Steiner tree with terminal set $S'$ are contained within $S$. In the case of $\Delta$-convexity, a set $S\subseteq V(G)$ is $\Delta$-\emph{convex} if every vertex $u\in V(G)\setminus S$ fails to form a triangle with any two vertices in $S$. 
Graph convexities have been studied in many contexts, including the determination of convexity invariants such as the hull number, the interval number, and the convexity number. A major reference work by Pelayo~\cite{pelayo-2015} provides an extensive survey of geodesic convexity.

A newly introduced convexity, known as \textit{cycle convexity}, was recently studied in~\cite{interval08}. 
In this convexity, for a set $S$ of vertices in a graph $G$, the \textit{cycle interval} of $S$, denoted by $\langle S \rangle$, is the set formed by the vertices of $S$ and any $w \in V(G)$ that form a cycle with the vertices of $S$. If $\langle S \rangle = S$, then $S$ is \textit{cycle convex} in $G$. The \textit{cycle convexity number} of $G$, $C_{cc}(G)$, is the maximum cardinality of a proper cycle convex set of $V(G)$. The \textit{cycle convex hull} of a set $S$, denoted by $\langle S \rangle_C$, is the smallest cycle convex set containing $S$. We say that a vertex $w \in V(G)$ is \textit{generated} by $S$ if $w\in\langle S \rangle_C$. The \textit{hull number} of $G$ in the cycle convexity, or more briefly, the \textit{cycle hull number} of $G$, $hn_{cc}(G)$, is the cardinality of the smallest set $S$ such that $\langle S \rangle_C = V(G)$. 

The study of cycle convexity in graphs has gained attention due to its applications in both graph theory and related fields such as Knot Theory~\cite{araujo2020cycle}. Concerning the cycle hull number, Araujo et al.~\cite{araujo2024hull} presented bounds for this parameter in $4$-regular planar graphs and proved that, given a planar graph $G$ and an integer $k$, determining whether $hn_{cc}(G) \leq k$ is NP-complete. They also showed that the parameter is computable in polynomial time for classes like chordal, $P_4$-sparse, and grid graphs. Regarding the cycle convexity number, Lima, Marcilon, and Medeiros~\cite{lima2024complexity} showed that, given a graph $G$ and an integer $k$, determining whether $C_{cc}(G) \geq k$ is both NP-complete and W[1]-hard when parameterized by the size of the solution. However, for certain graph classes such as extended $P_4$-laden graphs, they show an algorithm able to solve the problem in polynomial time. In addition, other results have been obtained for related parameters in cycle convexity, for instance, the interval number~\cite{interval08} and the percolation time~\cite{lima2024complexity}.

In this work, we explore some properties of cycle convex sets in graph products, focusing on the hull number and the convexity number. For the strong and lexicographic products of nontrivial connected graphs, we show that the cycle hull number is always two. For the Cartesian product, we establish tight bounds on the cycle hull number and present a closed formula for this parameter for Cartesian products of trees, which generalizes a known result for grid graphs~\cite{araujo2024hull}. On the complexity side, we close a gap in the literature regarding the cycle hull number for bipartite graphs, left open by Araujo et al.~\cite{araujo2024hull}. Specifically, given a graph $G$ and an integer $k$, we show that determining whether $hn_{cc}(G) \leq k$ is NP-complete even if $G$ is a bipartite Cartesian product graph.
Regarding the cycle convexity number, we show exact formulas for the Cartesian, strong, and lexicographic products of nontrivial connected graphs. As a corollary, we find the cycle convexity number of the three products when the factors are complete, cycle, or path graphs. In addition, the provided formulas directly imply the NP-completeness of the decision problem related to the convexity number for the three considered product graphs.

The paper is organized in more than three sections. Section~\ref{sec:pre} presents some notations and terminology, Section~\ref{sec:hull} focuses on the cycle hull number and Section~\ref{sec:cx} addresses the cycle convexity number.

\section{Preliminaries}\label{sec:pre}
We now introduce the terminology used throughout this paper. A graph $G = (V, E)$ is defined as a finite, undirected simple graph. The \emph{open neighborhood} $N(u)$ of a vertex $u \in V(G)$ is the set $\{ v \in V(G) : uv\in E(G) \}$, while the \emph{closed neighborhood} $N[u]$ is defined as $N(u) \cup \{ u \}$.  We will write $u \sim v$ if vertices $u$ and $v$ are adjacent. A vertex $v$ in a connected graph $G$ is a \emph{cut-vertex} if the graph $G - v$ is disconnected. The subgraph of $G$ induced by a subset $S \subseteq V(G)$ is denoted by $G[S]$. A vertex is \emph{simplicial} if its neighborhood induces a clique. The \emph{clique number} $\omega(G)$ of $G$ is the size of the largest clique in $G$, while the \emph{independence number} $\alpha(G)$ represents the size of the largest independent set. The path of order $\ell$ is denoted as $P_{\ell}$,  the cycle of length $\ell$ as $C_{\ell}$, the complete graph of order $\ell$ as $K_{\ell}$.


Let $G$ and $H$ be two graphs. In this paper, we discuss the cycle hull number and the cycle convexity number of the \emph{Cartesian product} $G \Box H$, \emph{lexicographic product} $G \circ H$, and \emph{strong product} $G \boxtimes H$. The graphs $G$ and $H$ are called \textit{factors}. All these products share the vertex set $V(G) \times V(H)$. For $(g_1, h_1), (g_2, h_2) \in V(G) \times V(H)$:
In the Cartesian product $G \square H$, the vertices $(g_1, h_1)$ and $(g_2, h_2)$ are adjacent if and only if either 
i) $g_1 \sim g_2$ in $G$ and $h_1 = h_2$, or 
ii) $g_1 = g_2$ and $h_1 \sim h_2$ in $H$.
In the lexicographic product $G \circ H$, these vertices are adjacent if either 
i) $g_1 \sim g_2$, or 
ii) $g_1 = g_2$ and $h_1 \sim h_2$.
Finally, in the strong product $G \boxtimes H$, the vertices $(g_1, h_1)$ and $(g_2, h_2)$ are adjacent if one of the following holds: 
i) $g_1 \sim g_2$ and $h_1 = h_2$, 
ii) $g_1 = g_2$ and $h_1 \sim h_2$, or 
iii) $g_1 \sim g_2$ and $h_1 \sim h_2$.

If $\ast \in \{\square, \circ, \boxtimes\}$, then the \emph{projection mappings} $\pi_G: V(G \ast H) \rightarrow V(G)$ and $\pi_H: V(G \ast H) \rightarrow V(H)$ are given by $\pi_G(u,v) = u$ and $\pi_H(u,v) = v$, respectively. We adopt the following conventions. For $u \in V(G)$ and $v \in V(H)$, we define $^uH$ to be the subgraph of $G \ast H$ induced by $\{ u \} \times V(H)$, which we call an \emph{$H$-layer}, while the \emph{$G$-layer}, $G^v$ is the subgraph induced by $V(G) \times \{ v \}$. 
If $S \subseteq V(G \square H)$, then the set $\{g \in V(G) :\ (g,h) \in S \text{ for some } h \in V(H)\}$ is the \emph{projection} $\pi_G(S)$ of $S$ on $G$. The projection $\pi_H(S)$ of $S$ on $H$ is defined analogously. Finally, for $S \subseteq V(G \ast H)$, we call $S$ a \emph{sub-product} of $G \ast H$ if $S = \pi_G(S) \times \pi_H(S)$.

In Section~\ref{sec:intro}, we introduced the cycle convexity definitions. Here, we revisit them, now specifically for the $P_3$-convexity, which will be useful for Section~\ref{sec:hull}. For a set $S$ of vertices in a graph $G$, the \textit{$P_3$-interval} of $S$, denoted by $\langle S \rangle^{{P_3}}$, is the set formed by the vertices of $S$ and any $w \in V(G)$ that forms a $P_3: u, w, v$ with $u,v \in S$. If $\langle S \rangle^{{P_3}} = S$, then $S$ is \textit{$P_3$-convex} in $G$. The \textit{$P_3$-convex hull} of a set $S$, denoted by $\langle S \rangle^{{P_3}}_C$, is the smallest $P_3$-convex set containing $S$. The \textit{$P_3$-hull number} of $G$, denoted as $hn_{P_3}(G)$, is the cardinality of the smallest set $S$ such that $\langle S \rangle^{{P_3}}_C = V(G)$. To avoid ambiguity, we will use the super/subscript $cc$ instead of $P_3$ when referring to cycle convexity. Additionally, we may include the super/subscripts $P_3(G)$ or $cc(G)$ to specify the graph $G$ under consideration.

\section{Hull Number}\label{sec:hull}
In this section, we investigate the cycle hull number of the Cartesian, strong, and lexicographic product of nontrivial connected graphs. We start by examining the strong and lexicographic products in Theorems~\ref{hullstrong} and~\ref{hulllexico}, respectively, for which it is proved that the cycle hull number is always two in each case.

\begin{theorem}\label{hullstrong}
    Let $G$ and $H$ be two nontrivial connected graphs. Then $$hn_{cc}(G\boxtimes H)=2.$$
\end{theorem}

\begin{proof}
Let $h_1h_2\in E(H)$ and $g\in V(G)$. We claim that $\{(g,h_1),(g,h_2)\}$ is a hull set of $G\boxtimes H$. By the definition of strong product, for any $g'\in N_G(g)$, the vertices $(g,h_1)$ and $(g,h_2)$ form triangles with both $(g',h_1)$ and $(g',h_2)$. Thus, $(g',h_1),(g',h_2)\in \langle\{(g,h_1),(g,h_2)\}\rangle_C$. This proves that $N_G(g)\times \{h_1,h_2\}\subseteq \langle\{(g,h_1),(g,h_2)\}\rangle_C$. Using similar arguments as above, $N_G(N_G(g))\times \{h_1,h_2\}$ is contained in $\langle N_G(g)\times \{h_1,h_2\}\rangle_C$. This sequentially shows that the vertices of the two $G$-layers, $G^{h_1}$ and $G^{h_2}$, are contained in $\langle\{(g,h_1),(g,h_2)\}\rangle_C$. Now $(g,h_1),(g',h_1)\in V(G^{h_1})\subseteq \langle\{(g,h_1),(g,h_2)\}\rangle_C$, where $g'\in N_G(g)$. Using the similar arguments as above, we get $V(^gH)\cup V(^{g'}H)\subseteq \langle (g,h_1),(g',h_1)\rangle_C$. This shows that $V(G\Box H)\subseteq\langle \{(g,h_1),(g,h_2)\}\rangle_C$ since both $G$ and $H$ are connected.     
\end{proof}

\begin{theorem}\label{hulllexico}
    Let $G$ and $H$ be two nontrivial connected graphs. Then $$hn_{cc}(G\circ H)=2.$$
\end{theorem}

\begin{proof} Let $(g,h_1)(g,h_2)$ be any edge in $G\circ H$. Then $\langle \{(g,h_1),(g,h_2)\}\rangle_C$ contains $N_G(g)\times V(H)$. This sequentially proves that $V(G\circ H)\subseteq\langle \{(g,h_1),(g,h_2)\}\rangle_C$ and so $hn_{cc}(G\circ H)=2.$ 
\end{proof}

Next, we turn our attention to the cycle hull number of Cartesian product graphs. Unlike the strong and lexicographic products, this is a more involved problem that required several intermediate steps to establish general bounds.
    
\begin{lemma}\label{union of subproduct}
 Let $G$ and $H$ be two nontrivial connected graphs. Then any convex set $S$ in $G\Box H$ is of the form $\displaystyle S=\bigcup_{i=1}^{r}(S_i\times T_i)$, where each $S_i\subseteq V(G)$ and $T_i\subseteq V(H)$. 
\end{lemma}
\begin{proof}
First suppose that $S$ is connected. In the following we prove that $S$ is a subproduct of $G\Box H$. 
 Specifically, for any \( (g,h), (g',h') \in V(G \Box H) \), we need to show that \( \{g,g'\} \times \{h,h'\} \subseteq S \). For that, consider a path \( P: (g,h) = (x_1, y_1), (x_2, y_2), \ldots,\\ (x_k, y_k) = (g', h') \) between \( (g,h) \) and \( (g',h') \) in the induced subgraph of \( S \). 
We will use induction on \( k \).
For the case \( k = 1 \), without loss of generality the path \( P \) is given by \( (g,h), (g',h), (g',h') \). Since \( S \) is convex, it follows that \( (g, h') \in S \). Therefore, the result holds for \( k = 1 \).
Assume that the result holds for all paths of length \( i < k \). That is, if \( (x,y) \in S \) and there exists a path between \( (g,h) \) and \( (x,y) \) in the induced subgraph of \( S \) in \( G \Box H \) of length at most \( k \), then \( \{g,x\} \times \{h,y\} \subseteq S \).
Given this assumption to the  path $P$ implies that \( \{x_1, x_2, \ldots, x_{k-1}\} \times \{y_1, y_2, \ldots, y_{k-1}\} \subseteq S \). Now, since $(x_{k-1}, y_{k-1})$ and $(x_k, y_k)$ are adjacent in $G\Box H$, Hence without loss of generality, we may assume that $y_k = y_{k-1}$. This in turn implies that $(x_k, y_{k-1}) \in S$. From the path $P$ in $G\Box H$, it is clear that \( \pi_H(P) : y_1, y_2, \ldots, y_{k-1} \) forms a $h-h'$ walk in \( H \). Now, choose any \( y_i \) where \( 1 \leq i \leq k-2 \). Let \( Q \) be a \( y_{k-1} - y_i \) subwalk of \( \pi_H(P) \), say \( Q : y_{k-1}, v_1, v_2, \ldots, v_r = y_i \), where \( \{v_1, v_2, \ldots, v_r\} \subseteq \{y_1, y_2, \ldots, y_{k-1}\} \).
By the induction hypothesis, \( \{x_{k-1}\} \times \{v_1, v_2, \ldots, v_r\} \subseteq S \). This sequentially shows that \( (x_k, v_2), (x_k, v_3), \ldots, (x_k, v_r) \in S \). Therefore, \( \{x_k\} \times \{y_1, y_2, \ldots, y_{k-1}\} \subseteq S \).
This ensures that \( \{g,g'\} \times \{h,h'\} \subseteq S \). 

Now, consider the case that the induced subgraph of $S$ contains more than one component. Let $S=C_1\cup C_2\cup \ldots \cup C_r$, where each $C_i$'s are pairwise disjoint components in the induced subgraph of $S$. By the definition of the cycle convexity, each component $C_i$ ($i=1,2,\ldots, r$) is convex and so from the first part of this theorem, each $C_i$ ($i=1,2,\ldots, r$) is a subproduct of $G\Box H$. i.e., for $i=1,2,\ldots r$, $C_i=S_i\times T_i$, for some $S_i\subseteq V(G)$ and $T_i\subseteq V(H)$. Therefore $\displaystyle S=\bigcup_{i=1}^{r}(S_i\times T_i)$.
\end{proof}

\begin{theorem}\label{theo:StimesT}
Let $G$ and $H$ be two nontrivial connected graphs and let $S$ and $T$ be any two convex sets in $G$ and $H$ respectively. Then $S\times T$ is convex in $G\Box H$. 
\end{theorem}

\begin{proof}
Let $S$ and $T$ be any two convex sets in $G$ and $H$ respectively. Assume to the contrary that $S\times T$ is not convex in $G\Box H$. Then we can find a vertex $(g,h)\in V(G\Box H)\setminus (S\times T)$ such that $(g,h)$ must have at least two distinct neighbours say $(g_1,h_1)$ and $(g_2,h_2)$ in the same connected component of the induced subgraph of $S\times T$. Then we have to consider the following four cases.\\
{\bf Case 1}: $g=g_1=g_2$, $hh_1,hh_2\in E(H)$. Since $g=g_1=g_2$, the three vertices $(g,h), (g_1,h_1),(g_2,h_2)$ are lie on a single $H$-layer $^gH$ and $h$ is adjacent to both $h_1$ and $h_2$. Since $h_1,h_2\in T$, $h\in T$. Therefore $(g,h)\in S\times T$.\\
{\bf Case 2}: $h=h_1=h_2$ and $gg_1,gg_2\in E(G)$.
Here we also get a contradiction by the similar argument as above.\\
{\bf Case 3}:  $g=g_1$, $h=h_2$, $hh_1\in E(H)$ and $gg_2\in E(G)$.     Since $g_1,g_2\in S$ and $h_1,h_2\in V(H)$, $(g,h)=(g_1,h_2)\in (\{g_1,g_2\}\times \{h_1,h_2\})\subseteq (S\times T)$. Therefore Case 3 is not possible.
\\{\bf Case 4}: $g=g_2$, $h=h_1$, $hh_2\in E(H)$ and $gg_1\in E(G)$. This case is also not possible by the similar arguments of Case 3.
\end{proof}

\begin{lemma}\label{projections}
    Let $G$ and $H$ be two connected nontrivial graphs. If $S$ is a hull set of $G\Box H$, then $\pi_G(S)$ and $\pi_H(S)$ are hull sets of $G$ and $H$ respectively.
\end{lemma}

\begin{proof}
    Let $S$ be a hull set of $G\Box H$. Then $S\subseteq \pi_G(S)\times \pi_H(S)\subseteq \langle \pi_G(S)\rangle_C \times \langle \pi_H(S)\rangle_C$. By Theorem~\ref{theo:StimesT} we get $\langle S\rangle_C = \langle \pi_G(S)\rangle_C \times \langle \pi_H(S)\rangle_C = V(G\Box H)$ and hence $\langle \pi_G(S)\rangle_C = V(G)$ and $\langle\pi_H(S)\rangle_C = V(H)$. 
\end{proof}

\begin{lemma}\label{line_column}
    Let $G$ and $H$ be two connected nontrivial graphs and $S \subseteq V(G \Box H)$. If $V(G^h) \cup V(^gH) \subseteq \langle S \rangle_C$, for some $g \in V(G)$ and $h \in V(H)$, then $S$ is a hull set of $G \Box H$.
\end{lemma}

\begin{proof}
We will prove for each vertex $g' \in N_G(g)$, $V(^{g'}H) \subseteq \langle S\rangle_C$. For each vertex $h' \in N_H(h)$, we have $(g',h')\in \langle \{(g,h),(g',h),(g,h')\}\rangle$. Therefore $(g',h') \in \langle S\rangle_C$. Now for any $h'' \in N_H(h')$, $(g',h'') \in \langle\{(g,h''), (g,h'),(g',h')\}\rangle$ and we get $(g',h'')\in \langle S\rangle_C$. By continuing this process, we obtain $V(^{g'}H)\subseteq \langle S\rangle_C$. Analogously, we can prove that for any $g'' \in N_G(g')$, $V(^{g''}H) \subseteq \langle S\rangle_C$. By following these steps repeatedly, we can conclude that $V(^{g'}H)\subseteq \langle S\rangle_C$, for all $g' \in V(G)$. Therefore $\langle S\rangle_C = V(G\Box H)$. 
\end{proof}

\begin{theorem}\label{hull1}
Let $G$ and $H$ be two nontrivial connected graphs. Then $$\max \{hn_{cc}(G), hn_{cc}(H),3\}\leq hn_{cc}(G\Box H)\leq hn_{cc}(G)+hn_{cc}(H)-1.$$
\end{theorem}

\begin{proof}
For the upper bound, Let $\{g_1,g_2,\ldots,g_r\}$ and $\{h_1,h_2,\ldots h_s\}$ be  two minimum hull sets of $G$ and $H$ respectively. Then $hn_{cc}(G)=r$ and $hn_{cc}(H)=s$. Consider the set $S=\{(g_1,h_1), (g_1,h_2),\ldots (g_1,h_s), (g_2,h_s),(g_3,h_s) \ldots ,(g_r,h_s)\}$. Since $\langle \{g_1,g_2,\ldots,g_r\}\rangle_C=V(G)$, $\langle\{(g_1,h_s), (g_2,h_s), \ldots ,(g_r,h_s)\}\rangle_C =V(G^{h_s})$. Also since $\langle\{h_1,h_2,\ldots, h_s\}\rangle_C=V(H)$, $\langle \{(g_1,h_1), (g_1, h_2),\ldots, (g_1,h_s)\} \rangle_C=V(^{g_1}H)$. 
Hence, by Lemma~\ref{line_column} we conclude that $\langle S\rangle_C =V(G\Box H)$. 
On the other hand, if $G$ and $H$ are two nontrivial connected graphs, then any 2-element subset of $V(G\Box H)$ is either an independent set of $G\Box H$ or its convex hull is contained in a single layer. This shows that any hull set in $G\Box H$ need at least three vertices. Hence the lower bound immediately follows from Lemma \ref{projections}.
\end{proof} 

Theorem \ref{hull1} immediately shows that $hn_{cc}(G\Box H)=3$ when both factors have hull number equals two. This applies, for instance, to complete graphs, where $hn_{cc}(K_m \Box K_n) = 3$.
When at least one of the factors has hull number equals two, some bounds are provided by the next theorem.

\begin{theorem}\label{theo:bounds_with_hnccH_two}
    Let $G$ and $H$ be two nontrivial connected graphs with $hn_{cc}(H)=2$. Then $hn_{cc}(G)\leq hn_{cc}(G\Box H)\leq hn_{cc}(G)+1$. Moreover $ hn_{cc}(G\Box H)= hn_{cc}(G)$ if and only if there is a minimum hull set $S$ in $G$ such that $S$ can be partitioned into two sets $S_1$ and $S_2$ with $\langle S_1 \rangle_C \cap \langle S_2 \rangle_C \neq\emptyset$.
\end{theorem}

\begin{proof}
Suppose $S= S_1 \cup S_2$  with $\langle S_1 \rangle_C \cap \langle S_2 \rangle_C \neq \emptyset$. Let $\{h,h'\}$ be a hull set of $H$ and let $g \in \langle S_1 \rangle_C \cap \langle S_2 \rangle_C$. Let $A_1,A_2, \dots, A_r$ be the components of $\langle S_1 \rangle_C$ and $B_1,B_2, \dots, B_s$ be the components of $\langle S_2 \rangle_C$. Assume without loss of generality that $g \in A_1 \cap B_1$. Let $T= ((S \setminus V(B_1)) \times \{h\}) \cup ((S \cap V(B_1)) \times \{h'\})$. In the following, we prove that $T$ is a hull set of $G \Box H$.

\smallskip
\noindent $\textbf{Claim 1: } V(G) \times \{h\} \subseteq \langle T \rangle_C$.

\smallskip
\noindent \textit{Proof of Claim~1.} Let $g' \in V(B_1) \setminus \{g\}$ and $P:g=g_1,g_2,...,g_k=g'$ be a $g-g'$ path in $B_1$. Then $(g_i,h') \in \langle T \rangle_C$ for all $i$ with $1 \leq i \leq k$. Since $(g_1,h),(g_1,h'),(g_2,h') \in \langle T \rangle_C$, we gain $(g_2,h) \in \langle T \rangle_C$. Similarly, given that $(g_2,h),(g_2,h'),(g_3,h') \in \langle T \rangle_C$, we obtain $(g_3,h) \in \langle T \rangle_C$. By continuing this way, we get $(g_k,h) \in \langle T \rangle_C$. This shows that $V(B_1) \times \{h\} \subseteq \langle T \rangle_C$. Thus $V(G) \times \{h\} \subseteq \langle T \rangle_C$. \hfill $\blacksquare$

\medskip
Since $g \in A_1 \cap B_1$, it follows that $(g,h),(g,h') \in \langle T \rangle_C$.
Given that $\{h,h'\}$ is a hull set of $H$,  we have $V(^gH) \subseteq \langle \{(g,h),(g,h')\} \rangle_C$. Furthermore, by Claim~1, $V(G^h) \subseteq \langle T \rangle_C$. Hence Lemma~\ref{line_column} implies that $T$ is a hull set of $G \Box H$. 



\medskip
Conversely, assume that $hn_{cc}(G \Box H)=hn_{cc}(G)$. Let $T$ be a hull set of $G \Box H$ of size $hn_{cc}(G)$. Since ${\pi}_G(T)$ is a hull set of $G$, it follows that $|\pi_G(T)|=|T|$. Define $T=(S_1 \times \{h_1\}) \cup (S_2 \times \{h_2\}) \cup...(S_k\times \{h_k\})$ and let $\{h,h'\}$ be a hull set of $H$. Also, consider  $S=\pi_G(T)=S_1\cup S_2 \cup...\cup S_k$, which is a hull set of $G$ of size $hn_{cc}(G)$.

\smallskip
\noindent $\textbf{Claim 2: }$  $S$ can be partitioned into $U \cup V$ such that $\langle U \rangle_C \cap \langle V \rangle_C \neq \emptyset.$

\smallskip
\noindent \textit{Proof of Claim~2.} Suppose by contradiction that $\langle S_i \rangle_C \cap \langle S_j \rangle_C = \emptyset$ for every distinct $i, j \in \{1, \dots, k\}.$
Consequently $\langle \cup_{j \neq i}(S_j \times \{h_j\}) \rangle_C \cap (V(G) \times \{h_i\})=\emptyset$.
This shows that $\langle S_i \times \{h_i\} \rangle_C=V(G) \times \{h_i\}$ and so $S_i$ is a hull set of $G$, a contradiction. Hence $\langle S_i \rangle_C \cap \langle S_j \rangle_C = \emptyset$ for some pair $i, j \in \{1, \dots, k\}$, $i \neq j$. Without loss of generality, fix $i=1$ and $j=2.$ Then $S=U \cup V$, where $U=S_1$ and $V=S_2 \cup S_3 \cup...\cup S_k$ with $\langle U \rangle_C \cap \langle V \rangle_C \neq \emptyset.$  \hfill $\blacksquare$

\smallskip
Claims~1 and~2 imply in the bounds $hn_{cc}(G)\leq hn_{cc}(G\Box H)\leq hn_{cc}(G)+1$.
\end{proof}

We now focus on the complexity of determining the hull number in cycle convexity for Cartesian product graphs. To this end, we first formally recall the associated decision problem as stated in \cite{araujo2024hull}.

\begin{problem}{\textsc{Hull Number in Cycle Convexity}}\\
\textbf{Instance:} A graph $G$ and a positive integer $k$.\\
\textbf{Question:} Is $hn_{cc}(G) \leq k$? 
\end{problem}

As far as we know, the complexity of \textsc{Hull Number in Cycle Convexity} for bipartite graphs was open~\cite{araujo2024hull}. Theorem~\ref{theo:bipartiteNP-complete} fills this gap in the literature and also is useful in Theorem~\ref{theo:cartesianNPComplete} to prove that the problem remains NP-complete even for bipartite Cartesian product graphs.

\begin{theorem}\label{theo:bipartiteNP-complete}
\textsc{Hull number in Cycle Convexity} remains NP-complete on bipartite graphs.
\end{theorem}

\begin{proof}
The problem is known to belong to NP~\cite{araujo2024hull}. For the hardness part, we perform a reduction from \textsc{Hull number in $P_3$ convexity} for bipartite graphs, which is NP-complete~\cite{araujo2013hull}. We first describe some subgraphs we use. Let $H(w)$ be the graph shown in Figure~\ref{fig:Hw_Fuv}(a). The \textit{non-edge} gadget $F^{uv}$ arises from the disjoint union of four vertices $u,u',v,v'$ together with five copies of $H(w)$, denoted as $H(w_i^{uv})$, for every $1 \leq i \leq 5$. Add to $E(F^{uv})$ the edges to make:
\begin{itemize}
    \item a $C_4: u, u', y_0(w_2^{uv}), y_0(w_1^{uv})$;
    \item a $C_4: v, v', y_0(w_4^{uv}), y_0(w_5^{uv})$; and
    \item a $P_3: u', y_0(w_3^{uv}),v'$.
\end{itemize}

A sketch of the graph $F^{uv}$, is depicted in Figure~\ref{fig:Hw_Fuv}(b). 

\begin{figure}[htb]
    \centering

\tikzset{every picture/.style={line width=0.75pt}} 

\begin{tikzpicture}[x=0.75pt,y=0.75pt,yscale=-1,xscale=1]

\draw [fill={rgb, 255:red, 255; green, 255; blue, 255 }  ,fill opacity=1 ]   (61.3,129.28) -- (112.64,129.16) ;
\draw [fill={rgb, 255:red, 255; green, 255; blue, 255 }  ,fill opacity=1 ]   (61.44,179.28) -- (112.77,179.16) ;
\draw [fill={rgb, 255:red, 255; green, 255; blue, 255 }  ,fill opacity=1 ]   (61.3,129.28) -- (61.44,179.28) ;
\draw [fill={rgb, 255:red, 255; green, 255; blue, 255 }  ,fill opacity=1 ]   (112.77,179.16) -- (112.64,129.16) ;
\draw [fill={rgb, 255:red, 255; green, 255; blue, 255 }  ,fill opacity=1 ]   (177.36,130.01) -- (228.69,129.89) ;
\draw [fill={rgb, 255:red, 255; green, 255; blue, 255 }  ,fill opacity=1 ]   (177.49,180.01) -- (228.82,179.88) ;
\draw [fill={rgb, 255:red, 255; green, 255; blue, 255 }  ,fill opacity=1 ]   (177.36,130.01) -- (177.49,180.01) ;
\draw [fill={rgb, 255:red, 255; green, 255; blue, 255 }  ,fill opacity=1 ]   (228.82,179.88) -- (228.69,129.89) ;
\draw [fill={rgb, 255:red, 255; green, 255; blue, 255 }  ,fill opacity=1 ]   (114.59,244.34) -- (147.39,210.62) ;
\draw [fill={rgb, 255:red, 255; green, 255; blue, 255 }  ,fill opacity=1 ]   (80.2,211.02) -- (112.77,179.16) ;
\draw [fill={rgb, 255:red, 255; green, 255; blue, 255 }  ,fill opacity=1 ]   (209.64,212.31) -- (177.49,180.01) ;
\draw [fill={rgb, 255:red, 255; green, 255; blue, 255 }  ,fill opacity=1 ]   (178.43,242.95) -- (147.39,210.62) ;
\draw [fill={rgb, 255:red, 255; green, 255; blue, 255 }  ,fill opacity=1 ]   (114.59,244.34) -- (80.2,210.13) ;
\draw [fill={rgb, 255:red, 255; green, 255; blue, 255 }  ,fill opacity=1 ]   (178.43,242.95) -- (208.7,212.37) ;
\draw [fill={rgb, 255:red, 255; green, 255; blue, 255 }  ,fill opacity=1 ]   (147.39,210.62) -- (112.77,178.26) ;
\draw [fill={rgb, 255:red, 255; green, 255; blue, 255 }  ,fill opacity=1 ]   (177.36,130.01) -- (146.13,98.39) ;
\draw [fill={rgb, 255:red, 255; green, 255; blue, 255 }  ,fill opacity=1 ]   (147.39,210.62) -- (177.49,179.12) ;
\draw [fill={rgb, 255:red, 255; green, 255; blue, 255 }  ,fill opacity=1 ]   (112.64,129.16) -- (144.25,96.49) ;
\draw [fill={rgb, 255:red, 255; green, 255; blue, 255 }  ,fill opacity=1 ]   (473.34,159.53) -- (540.33,159.32) ;
\draw [fill={rgb, 255:red, 255; green, 255; blue, 255 }  ,fill opacity=1 ]   (413.58,158.08) -- (445.1,194.74) ;
\draw [fill={rgb, 255:red, 255; green, 255; blue, 255 }  ,fill opacity=1 ]   (476.76,159.91) -- (445.1,194.74) ;
\draw [fill={rgb, 255:red, 255; green, 255; blue, 255 }  ,fill opacity=1 ]   (476.94,86.81) -- (540,86.65) ;
\draw [fill={rgb, 255:red, 255; green, 255; blue, 255 }  ,fill opacity=1 ]   (347.11,86.72) -- (413.11,86.64) ;
\draw [fill={rgb, 255:red, 255; green, 255; blue, 255 }  ,fill opacity=1 ]   (472.91,219.74) -- (445.1,194.74) ;
\draw [fill={rgb, 255:red, 255; green, 255; blue, 255 }  ,fill opacity=1 ]   (420.07,218.72) -- (445.1,194.74) ;
\draw [fill={rgb, 255:red, 255; green, 255; blue, 255 }  ,fill opacity=1 ]   (567.07,219.19) -- (540.62,192.84) ;
\draw [fill={rgb, 255:red, 255; green, 255; blue, 255 }  ,fill opacity=1 ]   (574.61,159.15) -- (601.06,185.5) ;
\draw [fill={rgb, 255:red, 255; green, 255; blue, 255 }  ,fill opacity=1 ]   (602.43,218.03) -- (567.07,219.19) ;
\draw [fill={rgb, 255:red, 255; green, 255; blue, 255 }  ,fill opacity=1 ]   (574.61,159.15) -- (540.09,158.89) ;
\draw [fill={rgb, 255:red, 255; green, 255; blue, 255 }  ,fill opacity=1 ]   (601.77,218.03) -- (601.06,185.5) ;
\draw [fill={rgb, 255:red, 255; green, 255; blue, 255 }  ,fill opacity=1 ]   (540.62,192.84) -- (540.09,158.89) ;
\draw [fill={rgb, 255:red, 255; green, 255; blue, 255 }  ,fill opacity=1 ]   (539.85,52.74) -- (566.58,26.68) ;
\draw [fill={rgb, 255:red, 255; green, 255; blue, 255 }  ,fill opacity=1 ]   (599.77,61.16) -- (573.03,87.22) ;
\draw [fill={rgb, 255:red, 255; green, 255; blue, 255 }  ,fill opacity=1 ]   (540.49,88.12) -- (539.85,52.74) ;
\draw [fill={rgb, 255:red, 255; green, 255; blue, 255 }  ,fill opacity=1 ]   (599.77,61.16) -- (600.53,26.65) ;
\draw [fill={rgb, 255:red, 255; green, 255; blue, 255 }  ,fill opacity=1 ]   (540.49,87.45) -- (573.03,87.22) ;
\draw [fill={rgb, 255:red, 255; green, 255; blue, 255 }  ,fill opacity=1 ]   (566.58,26.68) -- (600.53,26.65) ;
\draw [fill={rgb, 255:red, 255; green, 255; blue, 255 }  ,fill opacity=1 ]   (286.27,185.23) -- (312.64,158.79) ;
\draw [fill={rgb, 255:red, 255; green, 255; blue, 255 }  ,fill opacity=1 ]   (346.31,192.8) -- (319.95,219.23) ;
\draw [fill={rgb, 255:red, 255; green, 255; blue, 255 }  ,fill opacity=1 ]   (287.42,220.59) -- (286.27,185.23) ;
\draw [fill={rgb, 255:red, 255; green, 255; blue, 255 }  ,fill opacity=1 ]   (346.31,192.8) -- (346.59,158.28) ;
\draw [fill={rgb, 255:red, 255; green, 255; blue, 255 }  ,fill opacity=1 ]   (287.41,219.92) -- (319.95,219.23) ;
\draw [fill={rgb, 255:red, 255; green, 255; blue, 255 }  ,fill opacity=1 ]   (312.64,158.79) -- (346.59,158.28) ;
\draw [fill={rgb, 255:red, 255; green, 255; blue, 255 }  ,fill opacity=1 ]   (346.59,158.28) -- (413.58,158.08) ;
\draw [fill={rgb, 255:red, 255; green, 255; blue, 255 }  ,fill opacity=1 ]   (346.54,85.5) -- (346.36,158.6) ;
\draw [fill={rgb, 255:red, 255; green, 255; blue, 255 }  ,fill opacity=1 ]   (540,86.65) -- (540.33,159.32) ;
\draw [fill={rgb, 255:red, 255; green, 255; blue, 255 }  ,fill opacity=1 ]   (413.11,86.64) -- (413.58,158.08) ;
\draw [fill={rgb, 255:red, 255; green, 255; blue, 255 }  ,fill opacity=1 ]   (476.94,86.81) -- (476.76,159.91) ;
\draw [fill={rgb, 255:red, 255; green, 255; blue, 255 }  ,fill opacity=1 ]   (320.38,26.31) -- (346.72,52.77) ;
\draw [fill={rgb, 255:red, 255; green, 255; blue, 255 }  ,fill opacity=1 ]   (312.59,86.32) -- (286.25,59.86) ;
\draw [fill={rgb, 255:red, 255; green, 255; blue, 255 }  ,fill opacity=1 ]   (285.5,26.77) -- (320.38,26.31) ;
\draw [fill={rgb, 255:red, 255; green, 255; blue, 255 }  ,fill opacity=1 ]   (312.59,86.32) -- (347.11,86.72) ;
\draw [fill={rgb, 255:red, 255; green, 255; blue, 255 }  ,fill opacity=1 ]   (285.5,26.77) -- (286.25,59.86) ;
\draw [fill={rgb, 255:red, 255; green, 255; blue, 255 }  ,fill opacity=1 ]   (346.72,52.77) -- (347.11,86.72) ;
\draw  [color={rgb, 255:red, 0; green, 0; blue, 0 }  ,draw opacity=1 ][fill={rgb, 255:red, 255; green, 255; blue, 255 }  ,fill opacity=1 ] (351.36,158.83) .. controls (351.49,156.08) and (349.35,153.74) .. (346.6,153.61) .. controls (343.84,153.48) and (341.5,155.61) .. (341.37,158.37) .. controls (341.24,161.13) and (343.37,163.47) .. (346.13,163.6) .. controls (348.89,163.72) and (351.23,161.59) .. (351.36,158.83) -- cycle ;
\draw  [color={rgb, 255:red, 0; green, 0; blue, 0 }  ,draw opacity=1 ][fill={rgb, 255:red, 255; green, 255; blue, 255 }  ,fill opacity=1 ] (448.46,191.49) .. controls (446.59,189.47) and (443.42,189.34) .. (441.4,191.22) .. controls (439.37,193.09) and (439.24,196.25) .. (441.12,198.28) .. controls (442.99,200.31) and (446.16,200.43) .. (448.18,198.56) .. controls (450.21,196.68) and (450.34,193.52) .. (448.46,191.49) -- cycle ;
\draw  [color={rgb, 255:red, 0; green, 0; blue, 0 }  ,draw opacity=1 ][fill={rgb, 255:red, 255; green, 255; blue, 255 }  ,fill opacity=1 ] (540.47,154.33) .. controls (537.71,154.25) and (535.41,156.42) .. (535.33,159.18) .. controls (535.25,161.94) and (537.42,164.24) .. (540.18,164.32) .. controls (542.94,164.4) and (545.24,162.23) .. (545.32,159.47) .. controls (545.4,156.71) and (543.23,154.41) .. (540.47,154.33) -- cycle ;
\draw  [color={rgb, 255:red, 0; green, 0; blue, 0 }  ,draw opacity=1 ][fill={rgb, 255:red, 255; green, 255; blue, 255 }  ,fill opacity=1 ] (346.38,90.49) .. controls (349.14,90.59) and (351.45,88.42) .. (351.54,85.66) .. controls (351.63,82.9) and (349.46,80.59) .. (346.7,80.5) .. controls (343.94,80.41) and (341.63,82.57) .. (341.54,85.33) .. controls (341.45,88.09) and (343.62,90.4) .. (346.38,90.49) -- cycle ;
\draw  [color={rgb, 255:red, 0; green, 0; blue, 0 }  ,draw opacity=1 ][fill={rgb, 255:red, 255; green, 255; blue, 255 }  ,fill opacity=1 ] (535,86.45) .. controls (534.89,89.21) and (537.04,91.54) .. (539.8,91.65) .. controls (542.56,91.76) and (544.88,89.62) .. (545,86.86) .. controls (545.11,84.1) and (542.96,81.77) .. (540.2,81.66) .. controls (537.45,81.55) and (535.12,83.69) .. (535,86.45) -- cycle ;
\draw  [color={rgb, 255:red, 0; green, 0; blue, 0 }  ,draw opacity=1 ][fill={rgb, 255:red, 255; green, 255; blue, 255 }  ,fill opacity=1 ] (418.11,86.64) .. controls (418.11,83.88) and (415.87,81.64) .. (413.11,81.64) .. controls (410.34,81.64) and (408.11,83.88) .. (408.11,86.64) .. controls (408.11,89.4) and (410.34,91.64) .. (413.11,91.64) .. controls (415.87,91.64) and (418.11,89.4) .. (418.11,86.64) -- cycle ;
\draw  [color={rgb, 255:red, 0; green, 0; blue, 0 }  ,draw opacity=1 ][fill={rgb, 255:red, 255; green, 255; blue, 255 }  ,fill opacity=1 ] (481.94,86.81) .. controls (481.94,84.05) and (479.7,81.81) .. (476.94,81.81) .. controls (474.18,81.81) and (471.94,84.05) .. (471.94,86.81) .. controls (471.94,89.57) and (474.18,91.81) .. (476.94,91.81) .. controls (479.7,91.81) and (481.94,89.57) .. (481.94,86.81) -- cycle ;
\draw  [color={rgb, 255:red, 0; green, 0; blue, 0 }  ,draw opacity=1 ][fill={rgb, 255:red, 255; green, 255; blue, 255 }  ,fill opacity=1 ] (418.58,158.08) .. controls (418.58,155.32) and (416.34,153.08) .. (413.58,153.08) .. controls (410.82,153.08) and (408.58,155.32) .. (408.58,158.08) .. controls (408.58,160.84) and (410.82,163.08) .. (413.58,163.08) .. controls (416.34,163.08) and (418.58,160.84) .. (418.58,158.08) -- cycle ;
\draw  [color={rgb, 255:red, 0; green, 0; blue, 0 }  ,draw opacity=1 ][fill={rgb, 255:red, 255; green, 255; blue, 255 }  ,fill opacity=1 ] (481.76,159.91) .. controls (481.76,157.15) and (479.52,154.91) .. (476.76,154.91) .. controls (474,154.91) and (471.76,157.15) .. (471.76,159.91) .. controls (471.76,162.67) and (474,164.91) .. (476.76,164.91) .. controls (479.52,164.91) and (481.76,162.67) .. (481.76,159.91) -- cycle ;
\draw [fill={rgb, 255:red, 255; green, 255; blue, 255 }  ,fill opacity=1 ]   (419.5,256.05) -- (420.07,218.72) ;
\draw [fill={rgb, 255:red, 255; green, 255; blue, 255 }  ,fill opacity=1 ]   (472.91,219.74) -- (472.34,257.07) ;
\draw [fill={rgb, 255:red, 255; green, 255; blue, 255 }  ,fill opacity=1 ]   (445.91,280.67) -- (419.5,256.05) ;
\draw [fill={rgb, 255:red, 255; green, 255; blue, 255 }  ,fill opacity=1 ]   (445.91,280.67) -- (472.34,257.07) ;
\draw  [color={rgb, 255:red, 0; green, 0; blue, 0 }  ,draw opacity=1 ][fill={rgb, 255:red, 255; green, 255; blue, 255 }  ,fill opacity=1 ] (118.48,244.12) .. controls (118.58,241.97) and (116.92,240.14) .. (114.77,240.04) .. controls (112.62,239.94) and (110.79,241.61) .. (110.69,243.76) .. controls (110.59,245.91) and (112.25,247.73) .. (114.41,247.83) .. controls (116.56,247.93) and (118.38,246.27) .. (118.48,244.12) -- cycle ;
\draw  [color={rgb, 255:red, 0; green, 0; blue, 0 }  ,draw opacity=1 ][fill={rgb, 255:red, 0; green, 0; blue, 0 }  ,fill opacity=1 ] (182.32,243.13) .. controls (182.42,240.98) and (180.76,239.16) .. (178.61,239.06) .. controls (176.46,238.96) and (174.63,240.62) .. (174.53,242.77) .. controls (174.43,244.92) and (176.1,246.75) .. (178.25,246.85) .. controls (180.4,246.95) and (182.22,245.29) .. (182.32,243.13) -- cycle ;
\draw  [color={rgb, 255:red, 0; green, 0; blue, 0 }  ,draw opacity=1 ][fill={rgb, 255:red, 0; green, 0; blue, 0 }  ,fill opacity=1 ] (151.28,210.8) .. controls (151.38,208.65) and (149.72,206.82) .. (147.57,206.72) .. controls (145.41,206.62) and (143.59,208.29) .. (143.49,210.44) .. controls (143.39,212.59) and (145.05,214.42) .. (147.2,214.52) .. controls (149.36,214.62) and (151.18,212.95) .. (151.28,210.8) -- cycle ;
\draw  [color={rgb, 255:red, 0; green, 0; blue, 0 }  ,draw opacity=1 ][fill={rgb, 255:red, 255; green, 255; blue, 255 }  ,fill opacity=1 ] (212.6,212.55) .. controls (212.7,210.4) and (211.03,208.57) .. (208.88,208.47) .. controls (206.73,208.37) and (204.9,210.04) .. (204.8,212.19) .. controls (204.7,214.34) and (206.37,216.17) .. (208.52,216.27) .. controls (210.67,216.37) and (212.5,214.7) .. (212.6,212.55) -- cycle ;
\draw  [color={rgb, 255:red, 0; green, 0; blue, 0 }  ,draw opacity=1 ][fill={rgb, 255:red, 0; green, 0; blue, 0 }  ,fill opacity=1 ] (232.72,180.07) .. controls (232.82,177.91) and (231.16,176.09) .. (229.01,175.99) .. controls (226.85,175.89) and (225.03,177.55) .. (224.93,179.7) .. controls (224.83,181.86) and (226.49,183.68) .. (228.64,183.78) .. controls (230.8,183.88) and (232.62,182.22) .. (232.72,180.07) -- cycle ;
\draw  [color={rgb, 255:red, 0; green, 0; blue, 0 }  ,draw opacity=1 ][fill={rgb, 255:red, 0; green, 0; blue, 0 }  ,fill opacity=1 ] (181.39,180.19) .. controls (181.49,178.04) and (179.82,176.21) .. (177.67,176.11) .. controls (175.52,176.01) and (173.69,177.68) .. (173.59,179.83) .. controls (173.49,181.98) and (175.16,183.8) .. (177.31,183.9) .. controls (179.46,184) and (181.29,182.34) .. (181.39,180.19) -- cycle ;
\draw  [color={rgb, 255:red, 0; green, 0; blue, 0 }  ,draw opacity=1 ][fill={rgb, 255:red, 0; green, 0; blue, 0 }  ,fill opacity=1 ] (84.09,210.31) .. controls (84.19,208.16) and (82.53,206.33) .. (80.38,206.23) .. controls (78.23,206.13) and (76.4,207.79) .. (76.3,209.95) .. controls (76.2,212.1) and (77.86,213.92) .. (80.01,214.02) .. controls (82.17,214.12) and (83.99,212.46) .. (84.09,210.31) -- cycle ;
\draw  [color={rgb, 255:red, 0; green, 0; blue, 0 }  ,draw opacity=1 ][fill={rgb, 255:red, 0; green, 0; blue, 0 }  ,fill opacity=1 ] (116.67,179.34) .. controls (116.77,177.18) and (115.1,175.36) .. (112.95,175.26) .. controls (110.8,175.16) and (108.97,176.82) .. (108.87,178.97) .. controls (108.77,181.13) and (110.44,182.95) .. (112.59,183.05) .. controls (114.74,183.15) and (116.57,181.49) .. (116.67,179.34) -- cycle ;
\draw  [color={rgb, 255:red, 0; green, 0; blue, 0 }  ,draw opacity=1 ][fill={rgb, 255:red, 255; green, 255; blue, 255 }  ,fill opacity=1 ] (65.33,179.06) .. controls (65.43,176.91) and (63.77,175.08) .. (61.62,174.98) .. controls (59.47,174.88) and (57.64,176.55) .. (57.54,178.7) .. controls (57.44,180.85) and (59.1,182.68) .. (61.26,182.78) .. controls (63.41,182.88) and (65.23,181.21) .. (65.33,179.06) -- cycle ;
\draw  [color={rgb, 255:red, 0; green, 0; blue, 0 }  ,draw opacity=1 ][fill={rgb, 255:red, 0; green, 0; blue, 0 }  ,fill opacity=1 ] (65.2,129.47) .. controls (65.3,127.31) and (63.64,125.49) .. (61.49,125.39) .. controls (59.33,125.29) and (57.51,126.95) .. (57.41,129.1) .. controls (57.31,131.25) and (58.97,133.08) .. (61.12,133.18) .. controls (63.28,133.28) and (65.1,131.62) .. (65.2,129.47) -- cycle ;
\draw  [color={rgb, 255:red, 0; green, 0; blue, 0 }  ,draw opacity=1 ][fill={rgb, 255:red, 0; green, 0; blue, 0 }  ,fill opacity=1 ] (116.53,129.34) .. controls (116.63,127.19) and (114.97,125.36) .. (112.82,125.26) .. controls (110.67,125.16) and (108.84,126.83) .. (108.74,128.98) .. controls (108.64,131.13) and (110.31,132.96) .. (112.46,133.06) .. controls (114.61,133.16) and (116.43,131.49) .. (116.53,129.34) -- cycle ;
\draw  [color={rgb, 255:red, 0; green, 0; blue, 0 }  ,draw opacity=1 ][fill={rgb, 255:red, 255; green, 255; blue, 255 }  ,fill opacity=1 ] (148.15,96.67) .. controls (148.25,94.52) and (146.58,92.69) .. (144.43,92.59) .. controls (142.28,92.49) and (140.45,94.15) .. (140.35,96.3) .. controls (140.25,98.46) and (141.92,100.28) .. (144.07,100.38) .. controls (146.22,100.48) and (148.05,98.82) .. (148.15,96.67) -- cycle ;
\draw  [color={rgb, 255:red, 0; green, 0; blue, 0 }  ,draw opacity=1 ][fill={rgb, 255:red, 0; green, 0; blue, 0 }  ,fill opacity=1 ] (181.25,130.19) .. controls (181.35,128.04) and (179.69,126.22) .. (177.54,126.12) .. controls (175.39,126.02) and (173.56,127.68) .. (173.46,129.83) .. controls (173.36,131.98) and (175.03,133.81) .. (177.18,133.91) .. controls (179.33,134.01) and (181.15,132.35) .. (181.25,130.19) -- cycle ;
\draw  [color={rgb, 255:red, 0; green, 0; blue, 0 }  ,draw opacity=1 ][fill={rgb, 255:red, 255; green, 255; blue, 255 }  ,fill opacity=1 ] (232.59,130.07) .. controls (232.69,127.92) and (231.03,126.09) .. (228.87,125.99) .. controls (226.72,125.89) and (224.9,127.56) .. (224.8,129.71) .. controls (224.7,131.86) and (226.36,133.69) .. (228.51,133.79) .. controls (230.66,133.89) and (232.49,132.22) .. (232.59,130.07) -- cycle ;

\draw (59.55,106.83) node [anchor=north west][inner sep=0.75pt]    {$x_{1}( w)$};
\draw (131.38,71.43) node [anchor=north west][inner sep=0.75pt]    {$y_{0}( w)$};
\draw (127.51,30.51) node [anchor=north west][inner sep=0.75pt]    {$H( w)$};
\draw (63.37,157.68) node [anchor=north west][inner sep=0.75pt]    {$x_{2}( w)$};
\draw (86.11,200.43) node [anchor=north west][inner sep=0.75pt]    {$x_{3}( w)$};
\draw (71.16,243.45) node [anchor=north west][inner sep=0.75pt]    {$x_{4}( w)$};
\draw (180.43,241.85) node [anchor=north west][inner sep=0.75pt]    {$x_{5}( w)$};
\draw (214.34,205.39) node [anchor=north west][inner sep=0.75pt]    {$x_{6}( w)$};
\draw (229.98,158.98) node [anchor=north west][inner sep=0.75pt]    {$x_{7}( w)$};
\draw (229.94,110) node [anchor=north west][inner sep=0.75pt]    {$x_{8}( w)$};
\draw (114.47,122.25) node [anchor=north west][inner sep=0.75pt]    {$y_{1}( w)$};
\draw (114.7,157.56) node [anchor=north west][inner sep=0.75pt]    {$y_{2}( w)$};
\draw (152.3,201.04) node [anchor=north west][inner sep=0.75pt]    {$y_{3}( w)$};
\draw (178.07,108.47) node [anchor=north west][inner sep=0.75pt]    {$y_{5}( w)$};
\draw (179.88,157.58) node [anchor=north west][inner sep=0.75pt]    {$y_{4}( w)$};
\draw (347.7,89.23) node [anchor=north west][inner sep=0.75pt]    {$y_{0}\left( w_{1}^{uv}\right)$};
\draw (419,86.41) node [anchor=north west][inner sep=0.75pt]    {$u$};
\draw (483.67,89.07) node [anchor=north west][inner sep=0.75pt]    {$v$};
\draw (347.7,137.23) node [anchor=north west][inner sep=0.75pt]    {$y_{0}\left( w_{2}^{uv}\right)$};
\draw (453.04,181.9) node [anchor=north west][inner sep=0.75pt]    {$y_{0}\left( w_{3}^{uv}\right)$};
\draw (543.04,139.23) node [anchor=north west][inner sep=0.75pt]    {$y_{0}\left( w_{4}^{uv}\right)$};
\draw (542.85,87.47) node [anchor=north west][inner sep=0.75pt]    {$y_{0}\left( w_{5}^{uv}\right)$};
\draw (419.67,139.74) node [anchor=north west][inner sep=0.75pt]    {$u'$};
\draw (481,139.07) node [anchor=north west][inner sep=0.75pt]    {$v'$};
\draw (288.63,44.66) node [anchor=north west][inner sep=0.75pt]    {$H\left( w_{1}^{uv}\right)$};
\draw (543.3,45.66) node [anchor=north west][inner sep=0.75pt]    {$H\left( w_{5}^{uv}\right)$};
\draw (291.3,176) node [anchor=north west][inner sep=0.75pt]    {$H\left( w_{2}^{uv}\right)$};
\draw (419.3,224) node [anchor=north west][inner sep=0.75pt]    {$H\left( w_{3}^{uv}\right)$};
\draw (542.36,179.26) node [anchor=north west][inner sep=0.75pt]    {$H\left( w_{4}^{uv}\right)$};
\draw (431.67,29.51) node [anchor=north west][inner sep=0.75pt]    {$F^{uv}$};
\draw (98,29.61) node [anchor=north west][inner sep=0.75pt]   [align=left] {(a)};
\draw (403,29.61) node [anchor=north west][inner sep=0.75pt]   [align=left] {(b)};

\end{tikzpicture}
    \caption{Graphs (a) $H(w)$ and (b) $F^{uv}$. The black vertices in (a) represent a $cc$-hull set of the graph $H(w)$. The hexagons in (b) represent the subgraphs $H(w)$ in $F^{uv}$.}
    \label{fig:Hw_Fuv}
\end{figure}
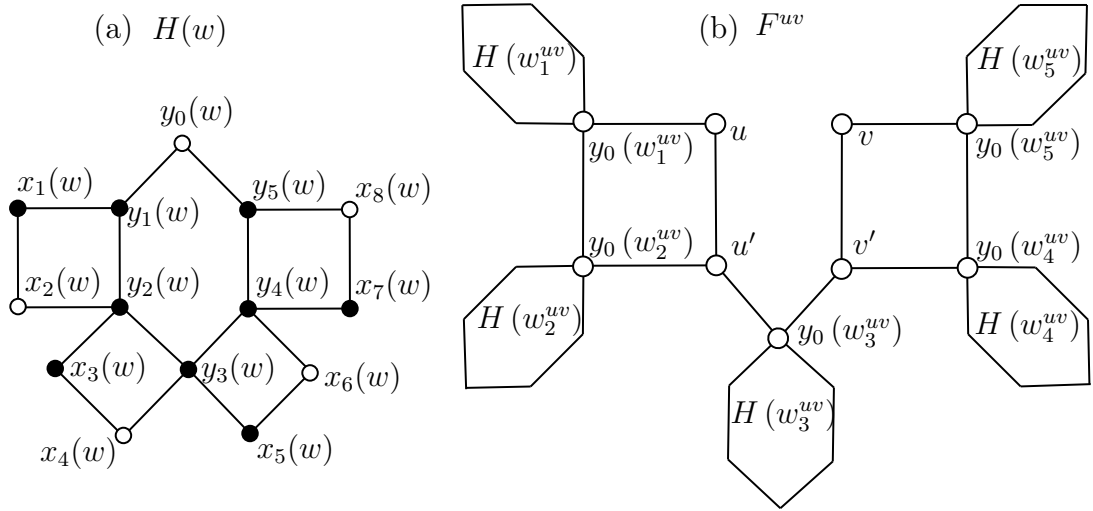

Let $(G,k)$ be an instance of \textsc{Hull number in $P_3$ convexity}, where $G$ is a bipartite graph and $k$ is a positive integer. We construct an instance $(G', k')$ of \textsc{Hull number in Cycle Convexity} as follows. Let $L = \{uv :$ $uv \notin E(G)$ and $N_G(u) \cap N_G(v) \neq \emptyset \}$ be the set of non-edges in $G$. 
The graph $G'$ arises from $G$ by adding a subgraph $F^{uv}$, for every $uv \in L$. We make $k' = k + 45|L|$. It is easy to see that $G'$ is a bipartite graph, since $G$ and $F^{uv}$ are clearly bipartite and, by definition of $L$, for every $uv \in L$, the connection between vertices in $V(F^{uv})$ and $V(G)$ make a $C_6: u,u', y_0(w_3^{uv}), v',v, z$, for every $z \in N_G(u) \cap N_G(v)$, which is bipartite. Further, the construction can be accomplished in polynomial time, since $|L| = O(|V(G)|)^2$. We show that $hn_{P_3}(G) \leq k$ if and only if $hn_{cc}(G') \leq k'$.

\smallskip
First, suppose that $hn_{P_3}(G) \leq k$. Let $S$ be a $P_3$-hull set of $G$ with $|S| \leq k$. For every $1 \leq i \leq 5$, we define $S(w_i^{uv}) = \{y_j(w_i^{uv}) : 1 \leq j \leq 5\} \cup \{ x_j(w_i^{uv}) : j \in \{1,3,5,7\} \}$, and, for every $uv \in L$, we define $\mathcal{S}^{uv} = S(w_1^{uv}) \cup \dots \cup S(w_5^{uv})$. We show that 
    $S' = S \cup \big(\bigcup_{uv \in L} \mathcal{S}^{uv}\big)$
is a $cc$-hull set of $G'$. Notice that $|S(w_i^{uv})| = 9$ and $\mathcal{S}^{uv} = 5|S(w_i^{uv})|$, then $|S'| = k + 45|L| = k'$. Let us proceed with some useful claims.

\medskip 
\noindent $\textbf{Claim 1: }$ Let $1 \leq i \leq 5$. Then $\langle S(w_i^{uv}) \rangle_C = V(H(w_i^{uv}))$.

\smallskip
\noindent \textit{Proof of Claim~1.} For short, we omit $(w_i^{uv})$ from the notation. By definition, $S = \{y_1, \dots, y_5, x_1, x_3,$ $x_5, x_7 \}$. The construction of $H$ implies that $y_0$ lies in a cycle $y_0,y_1, \dots y_5$; $x_2$ lies in a cycle $x_1, y_1, y_2, x_2$; $x_4$ lies in a cycle $x_3, y_2, y_3, x_3$; $x_6$ lies in a cycle $x_5,y_3,y_4,x_6$; and $x_8$ lies in a cycle $x_7,y_4,y_5,x_8$. Then, it holds that $y_0, x_2, x_4, x_6, x_8 \in \langle S \rangle_C$. \hfill $\blacksquare$



\medskip 
\noindent $\textbf{Claim 2: }$ Let $z \in V(G)$. If $z$ belongs to the $P_3$-convex hull of $S$ in $G$, then $z$ belongs to the $cc$-convex hull of $S'$ in $G'$.

\smallskip
\noindent \textit{Proof of Claim~2.} Here, for the $P_3$-convexity in $G$ and the cycle convexity in $G'$, we adopt the superscripts $P_3(G)$ and $cc(G')$, respectively, to clarify the convexity and the graph being considered in the interval $\langle \cdot \rangle$ and convex hull $\langle \cdot \rangle_C$ operations.
Suppose that $z \in \langle S \rangle^{_{P_3(G)}}_C$. If $z \in S$, then $z \in S'$ and there is nothing to do. So, suppose that $z \notin S$ and $z \notin S'$. Since $z \in \langle S \rangle^{_{P_3(G)}}_C$, there exist $u,v \in \langle S \rangle^{_{P_3(G)}}_C$ such that $z \in \langle u, v  \rangle^{_{P_3(G)}}$, that is, $z$ lies in a $P_3: u,z,v$ in $G$. We prove by induction that $z \in \langle S' \rangle^{_{cc(G')}}_C$. For the base case, let $u,v \in S$, which means $u, v \in S'$. Notice that $G$ is bipartite, then given that $z$ lies in a $P_3: u,z,v$ in $G$, it follows that $uv \notin E(G)$. Thus $uv \in L$ and, by construction, a non-edge gadget $F^{uv}$ exists in $G'$. Recall that Claim~1 implies $y_0(w_i^{uv}) \in \langle S' \rangle_C^{_{cc(G')}}$, for every $1 \leq i \leq 5$. Thus, given that $u,v \in S'$, we have that $u'$ lies in the cycle $u, u', y_0(w_2^{uv}), y_0(w_1^{uv})$ as well as $v'$ lies in the cycle $v, v', y_0(w_4^{uv}), y_0(w_5^{uv})$. Consequently $u',v' \in \langle S' \rangle^{_{cc(G')}}_C$. Now, since $z$ lies in a cycle $u,u',y_0(w_3^{uv}),v',v,z$, we obtain that $z \in \langle S' \rangle^{_{cc(G')}}_C$.

By inductive hypothesis, suppose that $u, v \in \langle S' \rangle^{_{cc(G')}}_C$ with either $u \notin S'$, $v \notin S'$, or both $u,v \notin S'$. Since $G$ is bipartite, $u,z,v$ does not induce a $K_3$ in $G$, which implies $uv \notin E(G)$ and $uv \in L$. Consequently, the existence of a non-edge gadget $F^{uv}$ with $u',v' \in \langle S' \rangle^{_{cc(G')}}_C$ ensures that $z \in \langle S' \rangle^{_{cc(G')}}_C$, as $z$ lies in the cycle $u,u',y_0(w_3^{uv}),v',v,z$. \hfill $\blacksquare$

\medskip 
By Claims~1 and 2, we conclude that $S'$ is a $cc$-hull set of $G'$ and then $hn_{cc}(G') \leq k'$. 

\medskip
For the converse, suppose that $hn_{cc}(G') \leq k'$. Let $S'$ be a $cc$-hull set of $G'$ with $|S'| \leq k' = k + 45|L|$. The subsequent claims will be helpful.

\medskip 
\noindent $\textbf{Claim 3: }$ Let $1 \leq i \leq 5$. It holds that: 
\begin{enumerate}
    \item[(i)] $|S' \cap V(H(w_i^{uv}))| \geq 9$; and 
    \item[(ii)] if $S'$ is a minimum $cc$-hull set of $G'$, then $y_0(w_i^{uv}) \notin S'$. 
\end{enumerate}

\smallskip
\noindent \textit{Proof of Claim~3.} For short, we omit $(w_i^{uv})$ from the notation. First, we examine some properties of $H$. Let us fix $X = \{x_i : 1 \leq i \leq 8\}$, $Y = \{y_i : 0 \leq i \leq 5\}$, and define the folowing cycles:
\begin{multicols}{3}
\begin{itemize}
    \item $A_1: x_1, y_1, y_2, x_2$;
    \item $A_2: x_3, y_2, y_3, x_4$;
    \item $A_3: x_6, y_4, y_3, x_5$;
    \item $A_4: x_8, y_5, y_4, x_7$;
    \item $A_5: y_0,y_1,\dots, y_5$.    
\end{itemize}
\end{multicols}

Let $1 \leq i \leq 4$. For every $a \in A_i \setminus Y$, notice that $|N_H(a) \cap Y| = 1$. Then, since $S'$ is a $cc$-hull set, it follows that $S' \cap (A_i \setminus Y) \neq \emptyset$, which means that $|S' \cap X| \geq 4$. Furthermore, for every $a \in A_5 \setminus \{y_0\}$, we have that $|N_H(a) \cap (X \cup \{y_0\})| = 2$, say $N_H(a) \cap (X \cup \{y_0\}) = \{a_1, a_2\}$. Since $G' \setminus Y$ is disconnected and $a_1$ and $a_2$ belong to distinct connected components, we have that $|S' \cap (A_5 \setminus \{y_0\})| \geq 1$.

(i) Suppose by contradiction that $S' \cap V(H) = \{v_1, \dots, v_8\}$. 
From the above, we assume that $v_1, \dots, v_4 \in S' \cap X$, and $v_5 \in S' \cap (A_5 \setminus \{y_0\})$. By construction, $A_i$ induces a $C_4$, $|A_i \cap A_j| = 1$, and all the paths from vertices in $A_i$ to vertices in $A_j$ pass through $Y$, for $i, j \in [4]$, $i \neq j$. This implies that $v_6, v_7, v_8$ generate at most three cycles in $\{ A_1, \dots, A_4 \}$, a contradiction.

(ii) Let $S'$ be a minimum $cc$-hull set of $G'$. By (i), we have that $|S' \cap V(H)| = 9$, with $|S' \cap X| \geq 4$ and $|S' \cap (A_5 \setminus \{y_0\})| \geq 1$. Suppose by contradiction that $y_0 \in S'$. We denote $S' \cap V(H) = \{y_0, v_1, \dots, v_8\}$.  We assume that $v_1, \dots, v_4 \in S' \cap X$, and $v_5 \in S' \cap (A_5 \setminus \{y_0\})$. By construction, $y_0$ lies only in cycle $A_5$, then, similarly to Item (i) $v_6, v_7, v_8$ generate at most three cycles in $\{ A_1, \dots, A_4 \}$, a contradiction. \hfill $\blacksquare$
 
\medskip 
\noindent $\textbf{Claim 4: }$ Let $uv \in L$. Then $|S' \cap F^{uv}| \geq 45$.

\smallskip
\noindent \textit{Proof of Claim~4.}
Let $1 \leq i \leq 5$. By Claim~3(ii), we know that $y_0(w_i^{uv})$ does not belong to a minimum $cc$-hull set of $G'$. Given that $y_0(w_i^{uv})$ is a cut-vertex in $G$ and $H(w_i^{uv})  \setminus \{ y_0(w_i^{uv}) \} $ is a connected component of $G' \setminus \{ y_0(w_i^{uv}) \} $, it follows that $|S' \cap F^{uv}| \geq |S' \cap (V(H(w_1^{uv}) \cup \dots \cup V(H(w_1^{uv})))| = 5 \cdot 9 = 45$. \hfill $\blacksquare$

\medskip
By Claim~4, we may assume that $|S' \cap (V(G) \cup \{u',v'\})| \geq k$. Notice that if $u' \in S'$, a $cc$-hull set $S''$ of $G'$ of the same order of $S'$ can be obtained by $S'' = S' \setminus \{u'\} \cap \{u\}$. The same reasoning applies to $v'$. We then assume that $S''$ is a $cc$-hull set of $G$ with $|S''| \geq k'$ and $S'' \cap \{u', v'\} = \emptyset$, for every $uv \in L$. We show that  $S'' \cap V(G)$ is a $P_3$-hull set of $G$. Recall that $y_0(w_i^{uv}) \in \langle S'' \rangle_C$, for every $1 \leq i \leq 5$. 

\medskip 
\noindent $\textbf{Claim 5: }$ Let $z \in V(G)$. If $z$ is in the $cc$-convex hull of $S''$ in $G'$, then $z$ is in the $P_3$-convex hull of $S = S'' \cap V(G)$ in $G$.

\smallskip
\noindent \textit{Proof of Claim~5.} As done in the proof of Claim~2, for the $P_3$-convexity in $G$ and the cycle convexity in $G'$, we use the superscripts $P_3(G)$ and $cc(G')$, respectively, to identify the convexity and the graph being considered in the interval  and convex hull operations.

Suppose that $z \in \langle S'' \rangle^{_{cc(G')}}_C$. If $z \in S''$, then $z \in S$ and there is nothing to do. So, suppose that $z \notin S''$. Since $z \in \langle S'' \rangle^{_{cc(G')}}_C$, there exist $u,v \in \langle S'' \rangle^{_{cc(G')}}_C$ such that $z \in \langle u, v  \rangle^{_{cc(G')}}$. We prove by induction that $z \in \langle S \rangle^{_{P_3(G)}}_C$. For the base case, let $u, v \in S$. Since $uv \notin E(G')$ and $uv \in L$, $z \in \langle S'' \rangle_C^{_{cc(G')}}$ because of the cycle $z, u, u', y_0(w_3^{uv}), v', v$ created by the non-edge gadget $F^{uv}$. Thus, the path with three vertices $u,z,v$ implies that $z \in \langle S \rangle_C^{_{P_3(G)}}$.

By inductive hypothesis, suppose that $u, v \in \langle S \rangle^{_{P_3(G)}}_C$, with either $u \notin S$, $v \notin S$, or both $u,v \notin S$. Since $G'$ is bipartite, we have that $uv \notin E(G')$. Consequently $uv \in L$, and by the non-edge gadget $F^{uv}$, we have $u',v' \in \langle S'' \rangle^{_{cc(G')}}_C$. Since $z$ lies in a cycle $u,u',y_0(w_3^{uv}),v',v,z$, we get $z \in \langle S'' \rangle^{_{cc(G')}}_C$. This, when mapped to the three-vertex path $u,z,v$ in $G$ implies that $z \in \langle S \rangle_C^{_{P_3(G)}}$. \hfill $\blacksquare$

\medskip
By Claim~5, we obtain that $S$ is a $P_3$-hull set of $G$ and then $hn_{P_3}(G) \leq k$ completing the proof. 
\end{proof}


We are now ready to prove the NP-completeness of \textsc{Hull Number in Cycle Convexity} for Cartesian product of graphs. 

\begin{theorem}\label{theo:cartesianNPComplete}
\textsc{Hull number in Cycle Convexity} remains NP-complete on Cartesian product $G_1 \square G_2$ graphs under either of the following conditions:  
\begin{enumerate}
    \item[(i)] $G_1 \square G_2$ is bipartite;  
    \item[(ii)] both $G_1$ and $G_2$ are planar.  
\end{enumerate}  
\end{theorem}

\begin{proof}
The NP-membership is already established~\cite{araujo2024hull}. 
For the NP-hardness, the proof of \textit{(i)} and \textit{(ii)} make use of the graph $H(w)$ shown in Figure~\ref{fig:Hw_Fuv}(a) in order to produce a graph that has a minimum hull set $S$ that can be partitioned into two $S_1 \cup S_2$ with $\langle S_1 \rangle_C \cap \langle S_2 \rangle_C \neq \emptyset$.

Let $\mathcal{B}$ and $\mathcal{P}$ be the classes of bipartite and planar graphs, respectively. From an instance $(G,k)$ of \textsc{Hull number in Cycle Convexity} in which $G \in \{ \mathcal{B}, \mathcal{P} \}$, we construct an instance $(G' \square K_2, k')$ of the same problem as described below. Before, let us remark that if $G \in \mathcal{B}$, the problem is NP-complete by Theorem~\ref{theo:bipartiteNP-complete} and if $G \in \mathcal{P}$, the NP-completeness follows from Theorem~4.2 by Araujo et al.~\cite{araujo2024hull}.

Let $H$ be the graph formed by two copies of $H(w)$ from Figure~\ref{fig:Hw_Fuv}(a), labeled $H(w_1), H(w_2)$, where the vertices $y_0(w_1)$ and $y_0(w_2)$ are identified as a single vertex $v$.  The graph $G'$ arises from the disjoint union of $G$ and $H$ by adding $uv \in E(G')$, for some vertex $u \in V(G)$. A sketch of the graph $G'$ is depicted in Figure~\ref{fig:graph_H}. The graph $G'$ is used to compute the Cartesian product $G' \square K_2$. Given that $H$ and $K_2$ are planar bipartite graphs, we have that: if $G \in \mathcal{B}$, $G' \square K_2$ is bipartite, as Item \textit{(i)}; and if $G \in \mathcal{P}$, both $G'$ and $K_2$ are planar, as Item \textit{(ii)}.
We show some properties of $G'$ in Claims~1 and~2. 

\begin{figure}[htb]
    \centering

\tikzset{every picture/.style={line width=0.75pt}} 

\begin{tikzpicture}[x=0.75pt,y=0.75pt,yscale=-1,xscale=1]

\draw [fill={rgb, 255:red, 255; green, 255; blue, 255 }  ,fill opacity=1 ]   (218.3,168.98) -- (243.05,140.94) ;
\draw [fill={rgb, 255:red, 255; green, 255; blue, 255 }  ,fill opacity=1 ]   (218.84,116.13) -- (243.05,140.94) ;
\draw [fill={rgb, 255:red, 255; green, 255; blue, 255 }  ,fill opacity=1 ]   (181.51,115.9) -- (218.84,116.13) ;
\draw [fill={rgb, 255:red, 255; green, 255; blue, 255 }  ,fill opacity=1 ]   (218.3,168.98) -- (180.97,168.75) ;
\draw [fill={rgb, 255:red, 255; green, 255; blue, 255 }  ,fill opacity=1 ]   (157.13,142.54) -- (181.51,115.9) ;
\draw [fill={rgb, 255:red, 255; green, 255; blue, 255 }  ,fill opacity=1 ]   (157.13,142.54) -- (180.97,168.75) ;
\draw [fill={rgb, 255:red, 255; green, 255; blue, 255 }  ,fill opacity=1 ]   (133.5,169.38) -- (158.25,141.34) ;
\draw [fill={rgb, 255:red, 255; green, 255; blue, 255 }  ,fill opacity=1 ]   (134.04,116.53) -- (158.25,141.34) ;
\draw [fill={rgb, 255:red, 255; green, 255; blue, 255 }  ,fill opacity=1 ]   (96.71,116.3) -- (134.04,116.53) ;
\draw [fill={rgb, 255:red, 255; green, 255; blue, 255 }  ,fill opacity=1 ]   (133.5,169.38) -- (96.17,169.15) ;
\draw [fill={rgb, 255:red, 255; green, 255; blue, 255 }  ,fill opacity=1 ]   (72.33,142.94) -- (96.71,116.3) ;
\draw [fill={rgb, 255:red, 255; green, 255; blue, 255 }  ,fill opacity=1 ]   (72.33,142.94) -- (96.17,169.15) ;
\draw    (157.23,142.13) -- (157.33,83) ;
\draw  [color={rgb, 255:red, 0; green, 0; blue, 0 }  ,draw opacity=1 ][fill={rgb, 255:red, 255; green, 255; blue, 255 }  ,fill opacity=1 ] (152.33,83) .. controls (152.33,80.24) and (154.57,78) .. (157.33,78) .. controls (160.09,78) and (162.33,80.24) .. (162.33,83) .. controls (162.33,85.76) and (160.09,88) .. (157.33,88) .. controls (154.57,88) and (152.33,85.76) .. (152.33,83) -- cycle ;
\draw  [color={rgb, 255:red, 0; green, 0; blue, 0 }  ,draw opacity=1 ][fill={rgb, 255:red, 255; green, 255; blue, 255 }  ,fill opacity=1 ] (152.23,142.13) .. controls (152.23,139.37) and (154.47,137.13) .. (157.23,137.13) .. controls (159.99,137.13) and (162.23,139.37) .. (162.23,142.13) .. controls (162.23,144.89) and (159.99,147.13) .. (157.23,147.13) .. controls (154.47,147.13) and (152.23,144.89) .. (152.23,142.13) -- cycle ;
\draw   (79.67,71.41) .. controls (79.67,57.09) and (114.26,45.48) .. (156.93,45.48) .. controls (199.61,45.48) and (234.2,57.09) .. (234.2,71.41) .. controls (234.2,85.73) and (199.61,97.33) .. (156.93,97.33) .. controls (114.26,97.33) and (79.67,85.73) .. (79.67,71.41) -- cycle ;
\draw   (62.8,108.08) -- (255.4,108.08) -- (255.4,179.08) -- (62.8,179.08) -- cycle ;

\draw (140,67.07) node [anchor=north west][inner sep=0.75pt]    {$u$};
\draw (151,148.73) node [anchor=north west][inner sep=0.75pt]    {$v$};
\draw (50.4,62.33) node [anchor=north west][inner sep=0.75pt]    {$G$};
\draw (93.19,132.93) node [anchor=north west][inner sep=0.75pt]    {$H( w_{1})$};
\draw (177.99,132.93) node [anchor=north west][inner sep=0.75pt]    {$H( w_{2})$};
\draw (30.4,132.73) node [anchor=north west][inner sep=0.75pt]    {$H$};

\end{tikzpicture}
    \caption{Graph $G'$. The hexagons represent the subgraphs $H(w)$ from Figure~\ref{fig:Hw_Fuv}(a).}
    \label{fig:graph_H}
\end{figure}
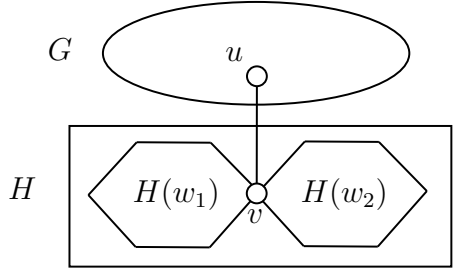

\smallskip 
\noindent $\textbf{Claim 1: }$ For every minimum hull set $S$ of $G'$, $|S \cap V(H)| = 18$.

\smallskip
\noindent \textit{Proof of Claim~1.} Let $S$ be a minimum hull set of $G'$.
Since $v$ is a cut-vertex in $G'$ and $v \notin S$, Claim~3 from the proof of Theorem~\ref{theo:bipartiteNP-complete} implies that $|S \cap V(H(w_i))| = 9$, for $i = 1,2$. Summing over $H(w_1)$ and $H(w_2)$ gives $|S \cap V(H)| = 18$. \hfill $\blacksquare$

\smallskip

For the next, let us fix $S(w_i) = \{y_j(w_i) : 1 \leq j \leq 5\} \cup \{ x_j(w_i) : j \in \{1,3,5,7\} \}$ for $i = 1, 2$.

\smallskip 
\noindent $\textbf{Claim 2: }$ There is a minimum hull set $S$ in $G'$ such that $S$ can be partitioned into $S_1 \cup S_2$ with $\langle S_1 \rangle_C \cap \langle S_2 \rangle_C \neq \varnothing$.

\smallskip
\noindent \textit{Proof of Claim~2.} Let $S$ be a minimum hull set of $G'$ such that $S(w_1) \cup S(w_2) \subseteq S$. By considering $S_1 = S(w_1)$ and $S_2 = S \setminus S(w_1)$, we have that 
$\langle S_1 \rangle_C  =  V(H(w_1))$, and $\langle S_2 \rangle_C  =  V(G') \cup V(H(w_2))$. Then $\langle S_1 \rangle_{C} \cap \langle S_2 \rangle_{C} = \{v\}$ as desired. \hfill $\blacksquare$

\medskip 
By Claim~2, Theorem~\ref{theo:bounds_with_hnccH_two} implies that $hn_{cc}(G' \square K_2) = hn_{cc}(G')$. So we just show that $hn_{cc}(G) \leq k$ if and only if $hn_{cc}(G') \leq k + 18$.
Let $S$ be a hull set of $G$ with $|S| \leq k$ and $S' = S \cup S(w_1) \cup S(w_2)$. Since $\langle S' \cap V(G)  \rangle_C = V(G)$ and $\langle S' \cap V(H)  \rangle_C = V(H)$, we obtain that $S'$ is a hull set of $G'$.
For the converse, let $S'$ be a hull set of $G'$ with $|S'| \leq k+18$. By Claim~1, we know that $|S' \cap V(H)| \geq 18$. Then, $|S' \cap V(G)| \leq k$. Since $v$ is a cut-vertex in $G'$ and $G$ is a connected component of $G' \setminus v$, any vertex in $V(G)$ must be generated by the convex hull of vertices in $V(G)$, consequently $S'\cap V(G)$ is a hull set of $G$.
\end{proof}

As a positive result, in the following, we prove that the cycle hull number of Cartesian product of any two trees can be computed in linear time. This finding generalizes the already known result for grid graphs, $hn_{cc}(P_m\Box P_n)=m+n-1$~\cite{araujo2024hull}.
\begin{theorem}
    Let $T_1$ and  $T_2$ be two nontrivial trees with orders $m$ and $n$ respectively. Then $hn_{cc}(T_1\Box T_2)=m+n-1.$
\end{theorem}
\begin{proof}
For a tree $T$ and an end vertex $v$ of $T,$ let $T^v$ be the tree obtained from $T$ by deleting the vertex $v$.

We prove the theorem by using induction on $|V(T_2)|$. For the case $|V(T_2)| = 2$, Theorem~\ref{theo:bounds_with_hnccH_two} implies that $h(T_1 \Box T_2)=h(T_1)+1 = |V(T_1)|+1$ for any tree $T_1$. Assume that $h(T_1 \Box T_2)=h(T_1)+h(T_2)-1$ for any trees $T_1$ and $T_2$ with $|V(T_2)|=k$.
Choose arbitrary trees $T_1$ and $T_2$ with $|V(T_2)|=k+1$. Since the Cartesian product is commutative, we may assume that $|V(T_1)| \geq |V(T_2)|$. We first prove the following claim.

\smallskip 
\noindent $\textbf{Claim 1: }$ $h(T_1\Box T_2^v) < h(T_1 \Box T_2)$.  

\smallskip
\noindent \textit{Proof of Claim~1.} Assume the contrary that $h(T_1\Box T_2^v) \geq h(T_1 \Box T_2)$. Let $S$ be a minimum hull set of $T_1 \Box T_2$. For each $x \in V(T_2)$, we fix $S_x=S \cap (T_1 \times \{x\})$. Now, for the support vertex $u$ of $v$ in $T_2$, we fix $U=\{x: (x,v)\in S_v\}$ and $S'= (S\setminus S_v) \cup (U \times \{u\})$.
Suppose that there is no edge between $S_u$ and $S_v$ in $T_1 \Box T_2$. Then, it is clear that $\langle S \rangle= \langle S \setminus S_v \rangle \cup S_v$ and $S_v \neq T_1 \times \{v\}$. This is a contradiction to the fact that $S$ is a hull set of $T_1 \Box T_2$. Hence, there exists at least one edge between $S_u$ and $S_v$. This, in turn implies that $|S'|<|S|$. Also it is straightforward to verify that $S'$ is a hull set of $T_1 \Box T_2^v$. This is a contradiction to the fact that $|S'|<|S|=h(T_1 \Box T_2) \leq h(T_1 \Box T_2^v)$. Hence the claim follows. \hfill $\blacksquare$

\smallskip
Now, the induction hypothesis and Claim 1 show that $h(T_1)+k-1=h(T_1\Box T_2^v)<h(T_1 \Box T_2) \leq h(T_1)+h(T_2)-1=h(T_1)+k. $ Thus, $h(T_1 \Box T_2)=h(T_1)+k$ and so the result follows for all trees $T_1$ and $T_2$. 
\end{proof}

\section{Convexity Number} \label{sec:cx}

We recall that the convexity number of a connected graph $G$ is the cardinality of a maximum proper convex set in $G$ and is denoted by $C_{cc}(G)$. If a graph $G$ contains a vertex $g$ that is not part of a cycle, then the convex hull of the remaining vertices will not generate $g$. Hence, $C_{cc}(G) = n-1$, which is the upper bound for any graph. Our first result establishes an equality for this parameter in the Cartesian product of two graphs.

\begin{theorem}\label{theo:cx}
    Let $G$ and $H$ be two nontrivial connected graphs with orders $m$ and $n$ respectively. Then $$C_{cc}(G\Box H) = \max\{n \cdot C_{cc}(G), m \cdot C_{cc}(H)\}.$$ 
\end{theorem}

\begin{proof}
Let $S_1$ be a proper convex set of $G$ with $|S_1| = C_{cc}(G)$ and set $T_1 = V(H)$. Define $S_2 = V(G)$ and let $T_2$ be a proper convex set of $H$ with $|T_2| = C_{cc}(H)$.
By Theorem~\ref{theo:StimesT}, $S_i \times T_i$ for $i = 1,2$ is convex in $G \Box H$. Then $C_{cc}(G\Box H) \geq \max\{|S_1 \times T_1|, |S_2 \times T_2|  \}$ and the lower bound follows. 
For the upper bound, suppose by contradiction that there exists a convex set $S$ in $G \square H$ with $|S| > \max\{n \cdot C_{cc}(G), m \cdot C_{cc}(H)\}$. We assume first that $n \cdot C_{cc}(G) \geq m \cdot C_{cc}(H)$.  Given that $|S| > n \cdot C_{cc}(G)$, there exists some $h \in V(H)$ such that $|S \cap V(G^h)| > C_{cc}(G)$. This implies that $V(G^h) \subset \langle S \cap V(G^h) \rangle_C$. Given that $|S| > m \cdot C_{cc}(H)$, there exists some $v \in V(G)$ such that $|S \cap V(^gH)| > C_{cc}(H)$. This gives that $V(^gH) \subset \langle S \cap V(^gH) \rangle_C$. So, given that $V(G^h) \cup V(^gH) \subset \langle S \rangle_C$, we have that $\langle S \rangle_C = V(G \Box H)$, which contradicts the assumption that $S$ is a proper convex set. Since the Cartesian product is commutative, the case $n \cdot C_{cc}(G) < m \cdot C_{cc}(H)$ follows by similar arguments.
\end{proof}

Given that $C_{cc}(K_n) = 1$, $C_{cc}(C_n) = n-2$, and $C_{cc}(T_n) = n-1$,  Theorem~\ref{theo:cx} implies the following corollary.

\begin{corollary}
    For positive integers $m$ and $n$, it holds that
    \begin{enumerate}[label=\arabic*.]
        \item $C_{cc}(K_m\Box K_n)= \max\{ m,n \}$.
        \item $C_{cc}(K_m\Box C_n)= \max\{ n, mn-2 \}$.
        \item $C_{cc}(K_m\Box T_n)= \max\{ n, mn-1 \}$.
        \item $C_{cc}(C_m\Box C_n)= \max\{ mn-2n, mn-2m \}$.
        \item $C_{cc}(C_m\Box T_n)= \max\{ mn-2n, mn-m \}$.
        \item $C_{cc}(T_m\Box T_n)=\max\{  mn-n, mn-m \}$.
    \end{enumerate}
\end{corollary}

Now, we proceed to strong and lexicographic products, where there is a relation with the independence number.

\begin{theorem}\label{theo:convexity_strong_lexico}
    Let $G$ and $H$ be two nontrivial connected graphs. For $\ast\in \{\boxtimes, \circ\},$ it holds that  $$C_{cc}(G\ast H)=\alpha(G\ast H).$$
\end{theorem}

\begin{proof} Let $\ast\in \{\boxtimes, \circ\}$. By Theorems~\ref{hullstrong} and~\ref{hulllexico}, any two adjacent vertices in $G\ast H$ forms a hull set in $G\ast H$. Thus the only proper convex sets in $G\ast H$ are independent sets. Therefore 
$C_{cc}(G\ast H)=\alpha(G\ast H)$. 
\end{proof}

It is easy to verify that $\alpha(K_n) = 1$, $\alpha(C_n) = \lfloor \frac{n}{2} \rfloor$,  $\alpha(P_n) = \lceil \frac{n}{2} \rceil$. Using the result $\alpha(K_n \boxtimes G) = \alpha(G)$ by Vesel~\cite{vesel1998independence}, the results for the strong product of cycles by Sonnemann and Krafft~\cite{sonnemann1974independence}, as well as the equality for the lexicographic product $\alpha(G \circ H) = \alpha(G)\alpha(H)$ by Geller and Stahl~\cite{geller1975chromatic}, we obtain the next corollary.

\begin{corollary}
    For positive integers $j, k, m$ and $n$, the following holds
    \begin{enumerate}[label=\arabic*.]
        \item $C_{cc}(K_m\boxtimes K_n) = C_{cc}(K_m\circ K_n)= 1$.
        \item $C_{cc}(K_m\boxtimes C_n)= C_{cc}(K_m\circ C_n)= \lfloor \frac{n}{2} \rfloor$.
        \item $C_{cc}(K_m\boxtimes P_n)= C_{cc}(K_m\circ P_n)= \lceil \frac{n}{2} \rceil$.
        \item $C_{cc}(C_m\circ C_n)= \lfloor \frac{m}{2} \rfloor \lfloor \frac{n}{2} \rfloor$.
        \item $C_{cc}(C_{2j}\boxtimes C_n)= j \lfloor \frac{n}{2} \rfloor$.
        \item $C_{cc}(C_{2j+1}\boxtimes C_{2k+1})= jk + \lfloor \frac{k}{2} \rfloor$, with $j \geq k$.
        \item $C_{cc}(C_m\boxtimes P_n)= C_{cc}(C_m\circ P_n)= \lfloor \frac{m}{2} \rfloor \lceil \frac{n}{2} \rceil$.
        \item $C_{cc}(P_m\boxtimes P_n)= C_{cc}(P_m\circ P_n)= \lceil \frac{m}{2} \rceil \lceil \frac{n}{2} \rceil$.
    \end{enumerate}
\end{corollary}

We close this section with a complexity result for the cycle convexity number derived from the above discussions. For that recall the related decision problem.

\begin{problem}{\textsc{Convexity Number in Cycle Convexity}}\\
\textbf{Instance:} A graph $G$ and a positive integer $k$.\\
\textbf{Question:} Is $C_{cc}(G) \geq k$?
\end{problem}

\begin{corollary}
\textsc{Convexity number in Cycle Convexity} is NP-complete when restricted to the class of Cartesian, strong, or lexicographic products of two nontrivial graphs.
\end{corollary}

\begin{proof}
The NP membership is clear. For the Cartesian product, given that \textsc{Convexity number in Cycle Convexity} is NP-complete~\cite{lima2024complexity}, the result is straightforward from Theorem~\ref{theo:cx}. For the strong and lexicographic products, recall that given a graph $G$ and a positive integer $k$, deciding whether $\alpha(G) \geq k$ is NP-complete~\cite{garey1979computers}. Hence the conclusion follows by  Theorem~\ref{theo:convexity_strong_lexico}.
\end{proof}

\bibliographystyle{amsplain}
\bibliography{cycle}
\end{document}